\documentclass[leqno,12pt]{amsart}
 
\usepackage{amsmath}
\usepackage{amsthm}  
\usepackage{verbatim}
\usepackage{enumerate} 
\usepackage{mathtools}  
\usepackage{amssymb}
\usepackage{mathrsfs}
\usepackage{float}
\usepackage[top=1.5in, bottom=1.5in, left=0.8in, right=0.8in]{geometry}  
\usepackage[colorlinks]{hyperref}
\hypersetup{
	colorlinks,
	citecolor=blue,
	linkcolor=red
}   
\numberwithin{equation}{section}
\usepackage{xypic}
\usepackage{bm}
\usepackage{tikz}

\newtheorem{theorem}{\textbf{Theorem}}[section]
\newtheorem{theorem*}{\textbf{Theorem}}
\newtheorem{definition}[theorem]{\textbf{Definition}}
\newtheorem{proposition}[theorem]{\textbf{Proposition}}

\newtheorem{problem}[theorem]{\textbf{Problem}}

\newtheorem{corollary}[theorem]{\textbf{Corollary}}
\newtheorem{remark}[theorem]{\textbf{Remark}}

\newtheorem{definition/proposition}[theorem]{\textbf{Definition/Proposition}}
\newtheorem{innercustomgeneric}{\customgenericname}
\providecommand{\customgenericname}{}
\newcommand{\newcustomtheorem}[2]{%
	\newenvironment{#1}[1]
	{%
		\renewcommand\customgenericname{#2}%
		\renewcommand\theinnercustomgeneric{##1}%
		\innercustomgeneric
	}
	{\endinnercustomgeneric}
}
\newcustomtheorem{customthm}{Theorem}
\newcustomtheorem{customprop}{Proposition}
\newcustomtheorem{customcoro}{Corollary}

\def\D{{\rm d}}

\DeclareMathOperator{\Tor}{Tor}
\title{Vanishing of Symplectic Homology and Obstruction to Flexible Fillability}
\author{Zhengyi Zhou}

\begin{document}
\maketitle
\begin{abstract}
	For any asymptotically dynamically convex contact manifold $Y$, we show that $SH_*(W)=0$ is a property independent of the choice of topologically simple (i.e.\ $c_1(W)=0$ and $\pi_{1}(Y)\rightarrow \pi_1(W)$ is injective) Liouville filling $W$. In particular, if $Y$ is the boundary of a flexible Weinstein domain, then any topologically simple Liouville filling $W$ has vanishing symplectic homology.  As a consequence, we answer a question of Lazarev partially: a contact manifold $Y$ admitting flexible fillings determines the integral cohomology of all the topologically simple Liouville fillings of $Y$.  The vanishing result provides an obstruction to flexible fillability. As an application, we show that all Brieskorn manifolds of dimension $\ge 5$ cannot be filled by flexible Weinstein manifolds.
\end{abstract}
\maketitle

\section{Introduction}
The study of the topology of fillings of a given contact manifold was initiated by Gromov and Eliashberg in the 1980s.  A combination of results in \cite{eliashberg1990filling,gromov1985pseudo} shows that the Weinstein filling of the standard contact $3$-sphere $(S^3,\xi_{std})$ is unique up to symplectic deformation.  In general, one can ask in what case contact manifolds can determine their Liouville fillings. For the topology of the fillings, Eliashberg-Floer-McDuff \cite{mcduff1991symplectic} proved that any Liouville filling of $(S^{2n-1},\xi_{std})$ is diffeomorphic to $D^{2n}$. Oancea-Viterbo \cite{oancea2012topology} showed that $H_*(Y) \to H_*(W)$ is surjective for a simply-connected subcritically-fillable contact manifold $Y$ and any Liouville filling $W$. Barth-Geiges-Zehmisch \cite{barth2016diffeomorphism} generalized the Eliashberg-Floer-McDuff theorem to the subcritical case: all Liouville fillings of a simply-connected subcritically-fillable contact manifold are diffeomorphic to each other. Regarding the symplectic topology of fillings, Seidel-Smith \cite{seidel2008biased} proved that any Liouville filling of $(S^{2n-1},\xi_{std})$ has vanishing symplectic homology, which provides evidence that the filling of $(S^{2n-1},\xi_{std})$ might be unique. In this note, we more generally consider asymptotically dynamically convex contact manifolds (Definition \ref{ADC}) and their topologically simple Liouville fillings (Definition \ref{topsimfilling}). We prove the following theorem in \S \ref{s3}.
\begin{theorem}\label{main}
	Let $(Y,\xi)$ be an asymptotically dynamically convex contact manifold of dimension $2n-1\ge 5$. Assume $Y$ has a Liouville filling $W$ that is topologically simple (i.e.\ $c_1(W)=0$ and $\pi_1(Y)\rightarrow \pi_1(W)$ is injective.) and has $SH_*(W;\mathbb{Z})=0$. Then $SH_*(W';\mathbb{Z})=0$ for every topologically simple Liouville filling $W'$.
\end{theorem}
Note that all the Liouville fillings of $(S^{2n-1},\xi_{std})$ are diffeomorphic to $D^{2n}$ by \cite{mcduff1991symplectic}, hence all Liouville fillings of $(S^{2n-1},\xi_{std})$ are topologically simple. Since $(S^{2n-1},\xi_{std})$ is asymptotically dynamically convex  and $SH_*(D^{2n};\mathbb{Z})=0$ for the standard filling $(D^{2n},\omega_{std})$. Theorem \ref{main} reproves the Seidel-Smith theorem \cite{seidel2008biased}. 

Recently, Lazarev \cite{lazarev2016contact} showed that all flexible fillings of a contact manifold have the same integral cohomology. He also raised the following question.
\begin{problem}[{\cite[\S 8 Problem 1]{lazarev2016contact}}]\label{p1}
	If $(Y,\xi)$ has a flexible filling $W$, do all Liouville fillings of $(Y,\xi)$ have the same cohomology as $W$? More generally, do all Liouville fillings of an arbitrary asymptotically
	dynamically convex $(Y,\xi)$ have the same cohomology?
\end{problem}	
Using Theorem \ref{main},  we answer Question \ref{p1} partially in the following Corollary.

\begin{corollary}\label{sub}
	Assume the contact manifold $(Y,\xi)$ has a flexible Weinstein filling $W$ and $c_1(\xi)=0$. If $W'$ is a topological simple Liouville filling of $Y$, then $SH_*(W';\mathbb{Z})=0$ and $H^*(W;\mathbb{Z})\cong H^*(W';\mathbb{Z})$. 
\end{corollary}
When $(Y,\xi)$ has a subcritical filling,  Corollary \ref{sub} shows that all the topological simple Liouville fillings have the same cohomology, which can also be derived from Barth-Geiges-Zehmisch's result \cite{barth2016diffeomorphism}. Barth-Geiges-Zehmisch's argument uses a foliation by holomorphic curves, while our method is algebraic and uses formal properties of symplectic homology. Moreover, our result is applicable to flexibly fillable contact manifold, which is not necessarily subcritically  fillable. The vanishing of the symplectic homology in Corollary \ref{sub} also provides evidence towards the following question.
\begin{problem}
	If $(Y,\xi)$ has a flexible filling $W$, is the Liouville filling of $Y$ unique?	
\end{problem}	

Corollary \ref{sub} also provides an obstruction to fillability by flexible Weinstein domains. We use this to prove that Brieskorn manifolds of dimension greater or equal to $5$ cannot be filled by flexible Weinstein domains. 

\section*{Acknowledgement}
I would like to thank my advisor Katrin Wehrheim for help with the expository, as well as continuous support and advice. I am also grateful to Oleg Lazarev for explaining his results and sharing his questions. This note is dedicated to the memory of Chenxue.

\section{Preliminaries on Fillings and Symplectic Homology}
\subsection{Symplectic fillings}
We assume throughout this note that $(Y,\xi)$ is a contact manifold of dimension $2n-1\ge 5$, such that $\xi$ is co-oriented and its first Chern class $c_1(\xi)=0$.
\begin{definition} \hspace{1em}
	\begin{itemize}
		\item $(W,\lambda)$ is a Liouville filling of $(Y,\xi)$ iff $\D \lambda$ is a symplectic form on $W$, the Liouville field $X_\lambda$, defined by $i_{X_\lambda} \D \lambda=\lambda$, is outward transverse along $\partial W$, and $(\partial W, \ker \lambda|_{\partial W})$ is contactomorphic to $(Y,\xi)$.  
		\item $(W,\lambda,\phi)$ is a Weinstein filling of $(Y,\xi)$ iff $(W,\lambda)$ is a Liouville filling, $\phi:W\rightarrow \mathbb{R}$ is a Morse function with maximal level set $\partial W$, and $X_\lambda$ is a gradient-like vector field for $\phi$.
		\item $(W,\lambda,\phi)$ is a flexible Weinstein filling of $(Y,\xi)$ iff $(W,\lambda,\phi)$ is a Weinstein filling, and there exist regular values  $c_1<\min \phi<c_2<\ldots<c_k<\max \phi$  of $\phi$, such that there are no critical values in $[c_k,\max \phi]$ and each $\phi^{-1}([c_i,c_{i+1}])$ is a Weinstein cobordism with a single critical point $p$ whose attaching sphere $\Lambda_p$ is either subcritical or a loose Legendrian in $(\phi^{-1}(c_i), \lambda|_{\phi^{-1}(c_i)})$, see \cite[Definition 11.29]{cieliebak2012stein} and \cite[Definition 2.4]{lazarev2016contact}
	\end{itemize}
	For simplicity, we only consider connected fillings in this note.
\end{definition}
For grading reasons in symplectic homology, we will only consider Liouville domains satisfying the following property.
\begin{definition}\label{topsimfilling}
	A Liouville domain $W$ is topologically simple iff  $c_1(W)=0$ and $\pi_1(\partial W)\rightarrow \pi_1(W)$ is injective.
\end{definition}
\begin{proposition}\label{top}
	Assume $(W, \lambda,\phi)$ is a Weinstein filling of $(Y,\xi)$ with $\dim Y \ge 5$ and $c_1(\xi)=0$. Then $W$ is topologically simple.
\end{proposition}	
\begin{proof}
	Since $\phi$ has only critical points with index less than or equal to $n=\frac{1}{2}\dim W$, the filling $W$ can be constructed from $Y$ by attaching handles with index greater than or equal to $n$. This implies that $\pi_2(W,Y)=0$, since $n\ge3$. Now  injectivity of $\pi_1(Y)\rightarrow \pi_1(W)$  follows from the long exact sequence
	$$\ldots \rightarrow \pi_2(W,Y)\rightarrow \pi_1(Y)\rightarrow\pi_1(W)\rightarrow \ldots.$$
	By \cite[Proposition 2.1]{lazarev2016contact}, $c_1(\xi)=0$ implies that $c_1(W)=0$.
\end{proof}	

Given two Liouville domains $W,W'$, we can join them by adding a 1-handle and then extend the Liouville structure to the handle. The resulting Liouville domain $W\natural W'$ is called the boundary connected sum of $W$ and $W'$. The contact boundary $\partial(W\natural W')$ is the contact connected sum $\partial W\# \partial W'$; see \cite[p.329]{cieliebak2012stein}.
\begin{proposition}\label{topsim}
	Let $W$ and $W'$ be two topologically simple Liouville domains. Then the boundary connected sum $W\natural W'$ is topologically simple.
\end{proposition}	
\begin{proof}
	Topologically, $W\natural W'$ is constructed by attaching a 1-handle to the disjoint union $W\coprod W'$. Let $\iota$ denote the inclusion $W\coprod W'\hookrightarrow W\natural W'$. Then $\iota^*: H^2(W\natural W';\mathbb{Z})\rightarrow H^2(W;\mathbb{Z})\oplus H^2(W';\mathbb{Z})$ is an isomorphism. Note that $c_1(W)=c_1(W')=0$ implies that $c_1(W\coprod W')=0$. Since $\iota^*c_1(W\natural W')=c_1(W\coprod W')$, we have $c_1(W\natural W')=0$.
	
	We have $\pi_1(W\natural W')=\pi_1(W)*\pi_1(W')$ and $\pi_1(\partial(W\natural W'))=\pi_1(\partial W\# \partial W')=\pi_1(\partial W)*\pi_1(\partial W')$ by the van Kampen theorem. Since $\pi_1(\partial W)\rightarrow \pi_1(W)$ and $\pi_1(\partial W')\rightarrow\pi_1(W')$ are both injective, the induced map on the free product of groups is also injective. Hence $W\natural W'$ is also topologically simple.
\end{proof}

The following definition is due to Lazarev  \cite[Definition 3.6]{lazarev2016contact}, and it generalizes the dynamically convex contact structures studied in \cite{abreu2014dynamical,cieliebak2015symplectic,hofer1999characterization}. A contact form $\alpha$ for the contact structure $\xi$ is called regular, if all Reeb orbits of $\alpha$ are non-degenerate. Let $\alpha_1,\alpha_2$ be contact forms for the same contact structure $\xi$. A partial order on the contact forms for a fixed co-oriented contact structure is given by $\alpha_1\ge \alpha_2$ iff $\alpha_1=f\alpha_2$ and $f\ge 1$;  see \cite[\S 3]{lazarev2016contact}.
\begin{definition}[{\cite[Definition 3.6]{lazarev2016contact}}]\label{ADC}
	A contact manifold $(Y^{2n-1},\xi)$ is called asymptotically dynamically convex, if there exist non-increasing regular contact forms $\alpha_1\ge\alpha_2\ldots $ for $\xi$ and a sequence of increasing  numbers $D_1<D_2<\ldots$ going to infinity such that every contractible Reeb orbit  $\gamma$ of $\alpha_k$ with period smaller than $D_k$ has the property that $\mu_{CZ}(\gamma)+n-3>0$, where $\mu_{CZ}(\gamma)$ is the Conley-Zehnder index of the non-degenerate orbit $\gamma$.
\end{definition}

An important fact is that the property of asymptotically dynamically convex is preserved under subcritical surgeries and flexible surgeries \cite[Theorem 3.15, 3.17, 3.18]{lazarev2016contact}.  In this note, we will use the following special case.

\begin{proposition}[{\cite[Theorem 3.15]{lazarev2016contact}}]\label{sum}
	Let $(Y_1,\xi_1)$, $(Y_2,\xi_2)$ be two asymptotically dynamically convex contact manifolds, then the contact connected sum $(Y_1\#Y_2,\xi_1\#\xi_2)$ is also asymptotically  dynamically convex. 
\end{proposition}
Yau \cite{yau2004cylindrical} showed that contact manifolds admitting subcritical Weinstein fillings are asymptotically dynamically convex. Lazarev generalized that to all the contact manifolds admitting flexible Weinstein fillings. 
\begin{proposition}[{\cite[Corollary 4.1]{lazarev2016contact}}]\label{flexADC}
	If $(Y,\xi)$ has a flexible Weinstein filling, then $(Y,\xi)$ is asymptotically dynamically convex. 
\end{proposition}

\subsection{Symplectic homology and positive symplectic homology}
To a Liouville filling $(W,\lambda)$ of the contact manifold $(Y,\xi)$, one can associate the completion $(\widehat{W}, \D \widehat{\lambda})=(W\cup_{Y}[1,\infty)_r \times Y, d\widehat{\lambda})$, where $\widehat{\lambda}=\lambda$ on $W$ and $\widehat{\lambda}=r(\lambda|_{Y})$ on $[1,\infty)_r\times Y$. For any coefficient ring or field $\bm{k}$, the symplectic homology $SH_*(W;\bm{k})$ is defined to be the Hamiltonian Floer homology on $(\widehat{W}, \D \widehat{\lambda})$ for the quadratic Hamiltonian $H$, where $H=r^2$ on the cylindrical end and $C^2$ small on $W$.  The chain complex is generated by the \textit{contractible} Hamiltonian periodic orbits, after a $C^2$ small perturbation of $H$. There are two types of generators: (1) critical points in $W$, (2) periodic orbits in $[0,\infty)\times Y$. One can choose the perturbation carefully such that there is a one-to-one correspondence between Reeb orbits  on $Y$ for the contact form $\lambda|_Y$ and pairs of Hamiltonian periodic orbits on $[0,\infty)\times Y$. The differential arises from counting the solutions to the Floer equations; see \cite{cieliebak2002handle,seidel2008biased,lazarev2016contact} for details of the construction. The complex generated by the critical points in $W$ is a subcomplex and the positive symplectic homology $SH^+_*(W;\bm{k})$ is  defined to be the homology of the quotient complex; see \cite{cieliebak2015symplectic}.  Moreover, we have the following exact triangle.
\begin{proposition}[\cite{seidel2008biased}]\label{triangle}
	For a Liouville domain $W$ with $c_1(W)=0$, there is an exact triangle:
	$$
	\xymatrix{H^{n-*}(W;\bm{k})\ar[rr] & &SH_*(W;\bm{k})\ar[ld]\\
		& \ar[lu]_{[-1]} SH^+_*(W;\bm{k}) &		
	}
	$$
\end{proposition}

Since there is a correspondence between the chain complex for $SH^+_*(W;\bm{k})$ and Reeb orbits, one may hope that $SH^+_*(W;\bm{k})$ is an invariant of the contact boundary. This is not true in general,  but if $Y$ is asymptotically dynamically convex, then the positive symplectic homology $SH^+_*(W,\bm{k})$ is independent of the choice of topologically simple Liouville filling $W$, hence a contact invariant.
\begin{proposition}[{\cite[Proposition 3.8]{lazarev2016contact}}]\label{independent}
	Let $(Y,\xi)$ be an asymptotically dynamically convex contact manifold. If there are two topologically simple Liouville fillings $W_1,W_2$, then $SH^+_*(W_1;\bm{k})\cong SH^+_*(W_2;\bm{k})$. 
\end{proposition}	

For the proof of theorem \ref{main}, we also need to recall the ring structure and boundary connected sum formula on symplectic homology. 
\begin{proposition}[{\cite[Theorem 6.6]{ritter2013topological}}]\label{ring}
	The pair of pants product $SH_*(W;\bm{k})\otimes SH_\star(W;\bm{k})\rightarrow SH_{*+\star-n}(W;\bm{k})$ makes $SH_*(W;\bm{k})$ into a unital ring, with the unit in $SH_n(W;\bm{k})$.   The morphism $H^{n-*}(W;\bm{k})\rightarrow SH_*(W;\bm{k})$  in Proposition \ref{triangle} is a unital ring morphism, where the multiplication on $H^{n-*}(W;\bm{k})$ is the usual cup product. 
\end{proposition}
\begin{remark}
	Our grading convention of $SH_*(W;\bm{k})$ follows the convention in \cite{bourgeois2012effect,cieliebak2015symplectic,lazarev2016contact}. \cite[Theorem 6.6]{ritter2013topological} was stated for the symplectic cohomology $SH^*(W;\bm{k})$ and it is related to the symplectic homology above by $SH^*(W;\bm{k})=SH_{n-*}(W;\bm{k})$. Also note that the grading conventions for symplectic cohomology in \cite{ritter2013topological} and \cite{seidel2008biased} are differed by a shift of $n$.
\end{remark}

\begin{proposition}[{\cite[Theorem 9.5]{ritter2013topological}}]\label{transfer}
	Let $\iota: (W',\lambda')\hookrightarrow (W,\lambda)$ be a Liouville subdomain \cite[\S 9]{ritter2013topological}, then the Viterbo transfer map $\iota^*_{SH}:SH_*(W;\bm{k}) \to SH_*(W';\bm{k})$ is a unital ring map.
\end{proposition}	

\begin{proposition}[{\cite[Theorem 1.11]{cieliebak2002handle}, \cite[Theorem 2]{fauck2017cieliebak}}]\label{boundarysum}
	Let $W,W'$ be two Liouville domains with $c_1(W)=c_1(W')=0$. Let $\iota:W\coprod W'\rightarrow W\natural W'$ denote the inclusion. Then the Viterbo transfer map $\iota_{SH}^*: SH_\star(W\natural W';\bm{k})\rightarrow SH_\star(W;\bm{k})\oplus SH_\star(W';\bm{k})$ arising from the inclusion $\iota$ is an isomorphism. 
\end{proposition}

\section{Proof of Theorem \ref{main}}\label{s3}
The proof of Theorem \ref{main} has some similarities with Smith's argument in \cite{seidel2008biased}, as both arguments use boundary connected sum. Smith's argument relies on Eliashberg-Floer-McDuff's result and only works for $(S^{2n-1},\xi_{std})$. Our argument uses the ring structure on symplectic homology.
\begin{proof}[Proof of Theorem \ref{main}]
	We will first show $SH_*(W';\bm{k})=0$ for any field $\bm{k}$. Note that the underlying chain complex defining $SH_*(W;\mathbb{Z})$ is a free abelian group generated by Hamiltonian periodic orbits. By universal coefficient theorem \cite[Theorem 3A.3]{hatcher2002algebraic},  $SH_*(W;\mathbb{Z})$ and $SH_*(W;\bm{k})$ are related by the following short exact sequence:
	\begin{equation}\label{universal}
	0\rightarrow SH_*(W;\mathbb{Z})\otimes_{\mathbb{Z}}\bm{k}\rightarrow SH_*(W;\bm{k}) \rightarrow \Tor\left(SH_{*-1}(W;\mathbb{Z}), \bm{k}\right)\rightarrow 0.
	\end{equation}		
	Hence $SH_*(W;\mathbb{Z})=0$ implies that $SH_*(W;\bm{k})=0$ for any field $\bm{k}$. Then Proposition \ref{triangle} yields an isomorphism
	\begin{equation}\label{SH+}
	SH^+_*(W;\bm{k})\cong H^{n+1-*}(W;\bm{k}).
	\end{equation}
	Since $(Y,\xi)$ is asymptotically  dynamically convex and both $W,W'$ are topologically simple, Proposition \ref{independent} and \eqref{SH+} yield an isomorphism
	\begin{equation}\label{SH'}
	SH^+_*(W';\bm{k})\cong SH_*^+(W;\bm{k})\cong H^{n+1-*}(W;\bm{k}).
	\end{equation}
	Applying Proposition \ref{triangle} to $W'$, we have the following long exact sequence:
	\begin{equation}\label{exact}
	\ldots \rightarrow H^{-1}(W';\bm{k})\rightarrow SH_{n+1}(W';\bm{k})\rightarrow SH^+_{n+1}(W';\bm{k})\rightarrow H^0(W';\bm{k})\rightarrow SH_n(W';\bm{k})\rightarrow \ldots,
	\end{equation}
	therefore $SH_{n+1}(W';\bm{k})\cong\ker\left(SH_{n+1}^+(W';\bm{k})\rightarrow H^0(W';\bm{k})\right)$.   By  \eqref{SH'},  $SH_{n+1}^+(W';\bm{k})$ is isomorphic to $H^0(W;\bm{k})\cong  \bm{k}$. Since $\bm{k}$ is a field, there are only two possibilities for $SH_{n+1}(W';\bm{k})$,
	\begin{equation}\label{SH} 
	SH_{n+1}(W';\bm{k})\cong 0 \text{ or } \bm{k}. 
	\end{equation}
	Next, we will rule out the case of $SH_{n+1}(W';\bm{k})=\bm{k}$. By Proposition \ref{sum},  the contact connected sum $Y\#Y$ is also asymptotically dynamically convex. By Proposition \ref{topsim}, both $W\natural W$ and $W'\natural W'$ are topologically simple Liouville fillings of $Y\# Y$. By Proposition \ref{boundarysum}, we have an isomorphism $SH_*(W\natural W;\mathbb{Z})\cong SH_*(W;\mathbb{Z})\oplus SH_*(W;\mathbb{Z})=0$.  Therefore the conditions in Theorem \ref{main} also hold for $Y\#Y$.  Hence the same argument for \eqref{SH} can be applied to $W'\natural W'$, and we conclude that $$ SH_{n+1}(W'\natural W';\bm{k})\cong0 \text{ or } \bm{k}.$$
	On the other hand, by Proposition \ref{boundarysum}, $SH_{n+1}(W'\natural W';\bm{k})$ is isomorphic to $SH_{n+1}(W';\bm{k})\oplus SH_{n+1}(W';\bm{k})$. Therefore the only possibility is 
	$$SH_{n+1}(W';\bm{k})=0.$$
	Consequentially, by the long exact sequence \eqref{exact}, $SH_{n+1}^+(W';\bm{k})\rightarrow H^0(W';\bm{k})$ is an injective map from $\bm{k}$ to $\bm{k}$. Since $\bm{k}$ is a field, we can conclude $SH_{n+1}^+(W';\bm{k})\rightarrow H^0(W';\bm{k})$ is an isomorphism. This is the reason why we want to work with field $\bm{k}$ instead of $\mathbb{Z}$. Then $H^0(W';\bm{k})\rightarrow SH_n(W';\bm{k})$ is zero by \eqref{exact}. By Proposition \ref{ring}, the morphism $H^0(W';\bm{k})\rightarrow SH_n(W';\bm{k})$ sends the unit in $H^0(W';\bm{k})$ to the unit of $SH_*(W';\bm{k})$, therefore the unit of $SH_*(W';\bm{k})$ is zero. This proves $SH_*(W'; \bm{k})=0$ for any field $\bm{k}$.
	
	Now assume that the integral symplectic homology $SH_*(W';\mathbb{Z})$ does not vanish. Since $H^{n-*}(W;\mathbb{Z})$ is a finitely generated abelian group, $SH_*^+(W;\mathbb{Z})=SH_*^+(W';\mathbb{Z})$ is also finitely generated by Proposition \ref{triangle}.  Proposition \ref{triangle} also implies that $SH_*(W';\mathbb{Z})$ is  finitely generated, since both $H^{n-*}(W';\mathbb{Z})$ and $SH_*^+(W';\mathbb{Z})$ are finitely generated. Thus, by the classification of finitely generated abelian groups, $SH_*(W';\mathbb{Z})$ either contains a $\mathbb{Z}$ summand or a  $\mathbb{Z}/m$ summand. In either case, there exists a prime $p$, such that $SH_*(W';\mathbb{Z})\otimes_{\mathbb{Z}} \mathbb{Z}/p\ne 0$. By \eqref{universal}, this implies $SH_*(W;\mathbb{Z}/p)\ne 0$, in contradiction to $SH_*(W;\bm{k})=0$ for $\bm{k}=\mathbb{Z}/p$.
\end{proof}	

\begin{corollary}\label{cor}
	Under the assumptions of Theorem \ref{main}, there is an isomorphism $H^*(W;\mathbb{Z})\cong H^*(W';\mathbb{Z})$.
\end{corollary}
\begin{proof}
	By Theorem \ref{main}, we have $SH_*(W;\mathbb{Z})=SH_*(W';\mathbb{Z})=0$. We also have $SH_*^+(W;\mathbb{Z})\cong SH_*^+(W';\mathbb{Z})$ by Proposition \ref{independent}. Then the long exact sequence in Proposition \ref{triangle} yields an isomorphism $H^{n+1-*}(W;\mathbb{Z})\cong SH^+_*(W;\bm{k})\cong SH^+_*(W';\bm{k})\cong H^{n+1-*}(W';\mathbb{Z})$.
\end{proof}

\begin{proof}[Proof of Corollary \ref{sub}]
	By Proposition \ref{top}, $W$ is a topologically simple filling. The contact manifold $(Y,\xi)$ is asymptotically dynamically convex by Proposition \ref{flexADC}. The flexible Weinstein domain $W$ has vanishing symplectic homology by \cite{bourgeois2012effect}. Then this corollary follows from  Theorem \ref{main} and Corollary \ref{cor}.
\end{proof}

In certain cases, cohomology group determines the diffeomorphism type. One of such examples is a simply connected $5$-manifold $Y^5$ with vanishing second Stiefel-Whitney class considered by Smale in \cite{smale1962structure}. Smale showed that $Y^5$ admits a $2$-connected filling $W^6$ and $W^6$ is unique up to the boundary connected sum with $(S^3\times S^3)\backslash D^6$. Moreover, all $2$-connected fillings have handle decompositions with only $0$ and $3$-handles. In particular, the diffeomorphism type of the $2$-connected filling $W^6$ is determined by the cohomology. Since $W^6$ admits an almost complex structure \cite[p.4]{lazarev2016contact}, $W^6$ admits a flexible Weinstein structure and $Y^5$ admits a contact structure $\xi$ by \cite[Theorem 1.5]{cieliebak2015symplectic}. Lazarev \cite[Corollary 1.3]{lazarev2016contact} proved all $2$-connected flexible fillings of $(Y^5,\xi)$ are diffeomorphic to each other. Using Corollary \ref{sub}, we can strengthen it to all $2$-connected Liouville fillings.
\begin{corollary}
	Let $(Y^5,\xi)$ be a simply-connected contact manifold with a $2$-connected flexible Weinstein filling. Then all $2$-connected Liouville fillings of $(Y^5,\xi)$ are diffeomorphic to each other.
\end{corollary}
Moreover, we have the following corollary for cobordisms.
\begin{corollary}\label{cob}
	Suppose $(Y,\xi)$ is simply-connected and has a flexible Weinstein filling. Assume there is a Liouville cobordism $E$ from $(Y_0,\xi_0)$ to $(Y,\xi)$ such that $c_1(E)=0$. Let $F$ be a Liouville filling of $(Y_0,\xi_0)$, such that $c_1(F)=0$  and $\iota_1^*-\iota_2^*: H^1(E)\oplus H^1(F) \rightarrow H^1(Y_0)$ is surjective, where $\iota_1, \iota_2: Y_0\rightarrow E,F$ are the inclusions. Then we have $SH_*(F;\mathbb{Z})=0$
\end{corollary}
\begin{proof}
	We can glue $E$ and $F$ to obtain a Liouville filling  $E\cup_{Y_0} F$  of  $(Y,\xi)$. Let $\tau_1:E\rightarrow E\cup_{Y_0}F, \tau_2:F\rightarrow E\cup_{Y_0}F$ denote the inclusions and consider the Mayer-Vietoris sequence
	$$\ldots \longrightarrow H^1(E)\oplus H^1(F) \stackrel{\iota_1^*-\iota_2^*}{\longrightarrow} H^1(Y_0) \longrightarrow H^2(E\cup_{Y_0} F)\stackrel{\tau_1^*\oplus \tau_2^*}{\longrightarrow}  H^2(E)\oplus H^2(F) \longrightarrow \ldots.$$
	Since $\iota_1^*-\iota_2^*: H^1(E)\oplus H^1(F)\rightarrow H^1(Y_0)$ is surjective, $\tau_1^*\oplus \tau_2^*:H^2(E\cup_{Y_0} F)\rightarrow  H^2(E)\oplus H^2(F)$ is injective. Note that $\tau_1^*(c_1(E\cup_{Y_0} F))=c_1(E)$ and $\tau_2^*(c_1(E\cup_{Y_0} F))=c_1(F)$, therefore $c_1(E)=c_1(F)=0$ implies that $c_1(E\cup_{Y_0} F)=0$. Since $\pi_1(Y)=0$, $E\cup_{Y_0}F$ is a topologically simple Liouville filling of $Y$. Since $Y$ has a flexible Weinstein filling by assumption, we have ${SH_*(E\cup_{Y_0} F;\mathbb{Z})}=0$ by Corollary \ref{sub}. Since the Viterbo transfer map $SH_*(E\cup_{Y_0} F;\mathbb{Z})\rightarrow SH_*(F;\mathbb{Z})$ is a unital ring map by Proposition \ref{transfer}, we obtain $SH_*(F;\mathbb{Z})=0$. 
\end{proof}	

Although Theorem \ref{main} is stated for asymptotic dynamical convex contact manifold admitting a topological simple Liouville filling with vanishing symplectic homology, the main examples are flexibly fillable contact manifolds. It is not clear to us whether there are other cases where Theorem \ref{main} can be applied. However, using Corollary \ref{cob}, we can show the vanishing of symplectic homology for a class of contact manifolds, which might not be flexibly fillable. Murphy-Siegel \cite{murphy2018subflexible} introduced the concept of subflexible manifolds, i.e. sublevel sets of flexible Weinstein domains and showed that subflexible domains are not necessarily flexible.

\begin{theorem}[{\cite[Theorem 1.4]{murphy2018subflexible}}]
	Let $X$ be a flexible Weinstein domain with $\dim X \ge 6$ and $c_1(X) = 0$. Then there is a Weinstein domain $X'$ such that
	\begin{itemize}
		\item $X'$ is a sublevel set of a flexible Weinstein domain, which is Weinstein deformation equivalent to $X$;
		\item $X'$ has nonvanishing twisted symplectic cohomology \cite[\S 3]{murphy2018subflexible}, hence $X'$ is not flexible. 
	\end{itemize} 
\end{theorem}

\begin{corollary}
	Let $X$ be a flexible Weinstein domain with $\dim X \ge 6$ and $\dim X \ne 8$ and $c_1(X) = 0$. Assume that $\partial X $ is simply connected. Let $X'$ be the nonflexible subflexible domain constructed in \cite[Theorem 1.4]{murphy2018subflexible}. If $F$ is a Liouville filling of $\partial X'$ with $c_1(F) = 0$, then $SH_*(F;\mathbb{Z}) = 0$.
\end{corollary}
\begin{proof}
	By the construction in the proof of \cite[Theorem 1.4]{murphy2018subflexible}, there is a Liouville cobordism $E$ from $\partial X'$ to $\partial X$ such that $c_1(E) = 0$. When $\dim X \ge 6$ and $\dim X \ne 8$, $E$ is constructed by attaching $k$-handles for $k\ge 3$ to $\partial X'$.\footnote{When $\dim X = 8$, $E$ has a $2$-handle that kills a generator of $H^1(\partial X')$.} Therefore $\iota_1^*:H^1(E) \to H^1(\partial X')$ is surjective. Then by Corollary \ref{cob},  we have $SH_*(F;\mathbb{Z}) = 0$.
\end{proof}

\begin{remark}
	It is not clear whether $\partial X'$ is flexibly fillable. But $\partial X'$ either provides an example of flexibly fillable contact manifold with different Weinstein fillings or is not flexibly fillable but vanishing result still holds.
\end{remark}

Corollary \ref{sub} also provides  the following obstruction to the existence of flexible Weinstein fillings.
\begin{corollary}\label{cri}
	Suppose $(Y,\xi)$ is a contact manifold. If there is a topologically simple Liouville filling $W'$ of $Y$  such that $SH_*(W';\mathbb{Z})\ne 0$. then there are no flexible Weinstein fillings for $(Y,\xi)$.
\end{corollary}	
In particular, if there is a Weinstein filling $W'$ of $Y$ such that $SH_*(W';\mathbb{Z}) \ne 0$, then there are no flexible Weinstein fillings for $(Y,\xi)$.

\begin{proof}[Proof of Corollary \ref{cri}]
	Assume otherwise that $Y$ has a flexible Weinstein filling $W$. Then by Corollary \ref{sub}, $SH_*(W';\mathbb{Z})=0$, which contradicts the condition. 
\end{proof}

\section{Applications in flexible fillability}
As an application of the obstruction considered in Corollary \ref{cri}, we show that all the Brieskorn manifolds \cite[Definition 2.7]{kwon2016brieskorn} with dimension greater or equal to $5$ cannot be filled by flexible Weinstein domains. We review some basics of Brieskorn manifolds following \cite{kwon2016brieskorn}. Let $a=(a_0,\ldots, a_n)$ be an $(n+1)$-tuple of integers, such that $a_i\ge 2$ for all $i=0,\ldots,n$.  The Brieskorn manifold $\Sigma(a)$ is defined as follows:
$$\Sigma(a):=\left\{z\in \mathbb{C}^{n+1}\bigg| z_0^{a_0}+\ldots+ z_n^{a_n}=0,\quad \sum_{j=0}^n|z_j|^2=1\right\}.$$
It has a contact form
\begin{equation}\label{contact}
\alpha_a=\frac{i}{8} \sum_{j=1}^na_j(z_j\D\overline{z}_j-\overline{z}_j\D z_j).\end{equation}
For $\epsilon$ sufficiently small, the Brieskorn manifold $(\Sigma(a),\xi_a:=\ker \alpha_a)$ has a Weinstein filling:
\begin{equation}\label{filling}V_\epsilon(a)=\left\{z\in \mathbb{C}^{n+1}\bigg| z_0^{a_0}+\ldots +z_n^{a_n}=\epsilon, \quad \sum_{j=0}^n|z_j|^2\le 1\right\}\end{equation}
with the  symplectic form induced from $\mathbb{C}^{n+1}$; see \cite[\S 2]{kwon2016brieskorn}.  Moreover, $V_{\epsilon}(a)$ is homotopy equivalent to $\bigvee^{\prod_{j=0}^n(a_j-1)}S^n$, the wedge of $\prod_{j=0}^n(a_j-1)$ spheres \cite[Proposition 3.2]{kwon2016brieskorn},  $\Sigma(a)$ is $(n-2)$-connected \cite[Proposition 3.5]{kwon2016brieskorn}.

Lazarev showed that the exotic spheres constructed  in \cite{ustilovsky1999infinitely}, which are 
$\Sigma(p,2,\ldots,2)\cong S^{4m+1}$ for $p\equiv \pm 1 \mod 8$, cannot be filled by flexible Weinstein domain using the fact that there exists $k \ge n+2$ such that $SH_k^+(V_\epsilon(p,2,\ldots,2))\ne 0$. Such argument can be applied to many Brieskorn manifolds, whose (positive) symplectic cohomology have been computed \cite{kwon2016brieskorn,uebele2016symplectic}.  The argument might be possible to generalize  to all Brieskorn manifolds by studying the Conley-Zehnder indexes carefully. However, in the following, we give a simple alternative proof for all Brieskorn manifolds using Corollary  \ref{cri}. 

\begin{theorem}\label{brieskorn}
	The Brieskorn manifolds of dimension $\ge 5$ cannot be filled by flexible Weinstein domains.
\end{theorem}	
\begin{proof}
	We have $H^2(V_{\epsilon}(a))=H^2\left(\bigvee^{\prod_{j=0}^n(a_j-1)} S^n\right)=0$, since $n\ge 3$ by assumption. Therefore the first Chern class $c_1(V_\epsilon(a))$  vanishes. Since $c_1(\xi_a)=c_1(V_{\epsilon}(a))|_{\Sigma(a)}$, we have $c_1(\xi_a)=0$.  By \cite[Theorem 1.2]{kwon2016brieskorn}, we know that $SH_{*}(V_\epsilon(a);\mathbb{Z})\ne 0$. By Corollary \ref{cri},  $\Sigma(a)$ cannot be filled by flexible Weinstein domains. 
\end{proof}	

Cieliebak-Eliashberg \cite[Theorem 17.2]{cieliebak2012stein} proved that for dimension $\ge 6$, every almost Weinstein domain $(W,J)$ with $ c_1(W,J)=0$ admits infinitely many non-symplectomorphic Weinstein structures $W_k:=\natural^k \mathbb{C}^n_M\natural W_0$, where $W_0$ is the flexible Weinstein domain in the homotopy class of the almost Weinstein structure $(W,J)$ and $\mathbb{C}^n_M$ is the exotic $\mathbb{C}^n$ constructed in \cite{mclean2009lefschetz}. Lazarev \cite[Remark 1.10]{lazarev2016contact} proved that the Cieliebak-Eliashberg-McLean contact manifolds $\partial W_k$ cannot be filled by flexible Weinstein domains except for possibly $\dim H^1(\partial W, \mathbb{Z}/2)+1$ values of $k\in \mathbb{N}$. By Proposition \ref{sum}, since $SH_*(\mathbb{C}^n_M)\ne 0$, we have $SH_*(W_k;\mathbb{Z})\cong SH_*(W_0;\mathbb{Z})\oplus_{i=1}^k SH_*(\mathbb{C}^n_M;\mathbb{Z})=\oplus_{i=1}^kSH_*(\mathbb{C}^n_{M};\mathbb{Z})\ne 0$. Then using Corollary \ref{cri}, we have the following more general result.

\begin{corollary}
	When $k\in \mathbb{N}^+$,  $\partial W_k$ cannot be filled by flexible Weinstein domains.
\end{corollary}

\bibliographystyle{plain} 
\bibliography{ref}
\end{document}